\documentclass[a4paper,11 pt]{amsart}

\RequirePackage{amsmath, amssymb, amsthm, amsfonts}
\usepackage{fullpage}
\newtheorem{theo}{Theorem}[section]
\newtheorem{lem}[theo]{Lemma}
\newtheorem{prop}[theo]{Proposition}

\newtheorem{defin}[theo]{Definition}

\theoremstyle{definition}

\newtheorem{rem}[theo]{Remark}

\usepackage[all]{xy}
\usepackage{xcolor}
\usepackage{comment}

\newcommand{\void}[1]{}     

\newcommand{\lra}{\longrightarrow}

\newcommand{\PP}{\mathbb{P}}

\newcommand{\PGL}{{\mathrm{PGL}}}

\newcommand{\Aut}{{\mathrm{Aut}}}

\newcommand{\Gal}{{\mathrm{Gal}(\overline{\mathbb{Q}}/\mathbb{Q)}}}

\newcommand{\ZZ}{\mathbb{Z}}

\newcommand{\QQ}{\mathbb{Q}}
\newcommand{\CC}{\mathbb{C}}

\title{A faithful action of Gal($\overline{\QQ}/\QQ$) on Zariski multiplets}
\author{Michael L\"onne and Matteo Penegini}
\address{Michael L\"onne\\  
Mathematik VIII, Universit\"at Bayreuth,
NWII, Universit\"atsstrasse 30, D-95447 Bayreuth,  Germany}
\email{michael.loenne@uni-bayreuth.de}
\address{Matteo Penegini\\
Universit\`a degli Studi di Pavia\\
Dipartimento di Matematica - " F. Casorati"\\
Via Ferrata 5,
I-27100 Pavia (PV) - 
Italy} 
\email{matteo.penegini@unipv.it}

\begin{document}


\maketitle


\begin{abstract}

In this work, we establish two main results in the context of arithmetic and geometric properties of plane curves. First, we construct numerous new examples of arithmetic Zariski pairs and multiplets, where only a few ones were previously available. Second, we describe a faithful action of the absolute Galois group on the equisingular strata of plane curves, providing insights into the interplay between Galois representations and the geometry of singular plane curves. We conclude the paper with very concretes examples of the general results obtained. 
\end{abstract}


\section{Introduction}
The purpose of this work is twofold. First, we aim to expand the known landscape of arithmetic Zariski pairs and multiplets by providing a substantial number of new examples,  enriching a domain where only a few ones were previously available \cite{S08,AC17}. Second, we establish a faithful action of the absolute Galois group on the equisingular strata of plane curves, revealing an intriguing connection between arithmetic structures and singularity theory.  Both results can be seen as a natural continuation of our previous work \cite{LP20} adding the new results obtained in \cite{BCG06,BCG15}. 

The notion of \emph{Zariski pair} was introduced by E. Artal in his seminal paper 
\cite{artal94}.

\begin{defin}\label{def_ZP} A pair of complex reduced plane curves  $(B_1, B_2)$ is called a Zariski pairs if it satisfies the following conditions:
\begin{enumerate}
\item There exist tubular neighborhoods $T(B_i)$, $i = 1, 2$, and a homeomorphism $h\colon T(B_1) \lra T(B_2)$ such that $h(B_1) = B_2$.
\item There exists no homeomorphism $f : \PP^2 \lra \PP^2$ with $f(B_1) = B_2$ i.e., the pairs $(\PP^2,B_1)$ and $(\PP^2,B_2)$ are not homeomorphic.
\end{enumerate}
Analogously $(B_1 ,\ldots , B_k)$ is a Zariski $k$-plet if $(B_i, B_j)$ is a Zariski pair for
any i $\neq j$.
\end{defin} 

As noted in \cite[Remark 3]{ACT10}) the first condition in Definition \ref{def_ZP} can be replaced
by equality of combinatorial data associated to the $B_i$, $i = 1, 2$,
which in case of irreducible curves consist of the degree and the
counting function associating to each topological type of singularities
the number of corresponding singular points.

This definition put a new focus on an important question for plane complex projective curves:
\begin{quote}
Do the topological types of the singularities of a plane curve $B$ and the degree
determine the homeomorphism type of the pair $(\PP^2,B)$?
\end{quote}
Obviously this question has positive answer in case of plane curves without 
singularities, but already the case of nodal curves took a long time to be
established beyond doubt, cf.\ Severi \cite{S21}, Harris \cite{H86}.
In these, and many other cases the proof shows the connectedness of the
locus of curves with given singularity types in the projective space of plane
curves of the given degree. 

Instead, Zariski proved that in degree $6$ the complement of a sextic, regular
except for six cusps, has non-abelian fundamental group only if the cusps
lie on a conic, so to honour his results \emph{Zariski pairs} were named.

The excellent survey \cite{ACT10} can report on many new examples
and provides some strategy to find additional ones in two steps
\begin{enumerate}
\item[(I)]
Locate curves of the same invariants in different connected components
of the corresponding locus,
\item[(II)]
Find an effective invariant of the embedded pair, which distinguishes the
curves.
\end{enumerate}

In our previous work \cite{LP20} we used the fundamental group of the complement together with the
conjugacy class of meridians to the curve, which is well-defined for
irreducible curves, in the second step.

Addressing the first step,
the  idea developed in \cite{LP20}  is to exploit the disconnectedness of moduli
of surfaces of general type isogenous to a product.
As these surface have ample canonical class with $K_S^2\geq8$, the $m$-canonical map
is an embedding into a projective space $\PP^N$
at least for $m\geq 3$ by Bombieri's theorem \cite{B73}.

An embedded surface then is mapped onto $\PP^2$ using projection
with center a disjoint projective subspace of codimension $3$.
By the theorem of Ciliberto and Flamini \cite{CiF11} the corresponding branch
curve is reduced, irreducible and smooth except for ordinary nodes
and cusps if the center is sufficiently general.

This procedure extends nicely to families of 
$m$-canonical embeddings and
centers and thus associates a connected component of equisingular
plane irreducible curves to a connected component of surfaces of
general type with ample canonical bundle.

In this work we go even further. Inspired by and relying on the result given in \cite{BCG15} we give new insight on arithmetic Zariski pairs. We define arithmetic Zariski pairs following Shimada 

\begin{defin} A Zariski pair $(B_1, B_2)$ is called \emph{arithmetic Zariski pair} if $B_1$ and $B_2$ are conjugate plane curves, i.e.\ there exists a polynomial $P(x_0,x_1,x_2)$ with complex coefficients and an automorphism $\sigma$ of the field $\CC$ such that 
\[
B_1=\{P=0\}, \textrm{ and } B_2=\{P^{\sigma}=0\}
\]
where $P^{\sigma}$ is the polynomial obtained from $P$ by applying $\sigma$ to its coefficients. 
\end{defin}

Our main theorem on Zariski multiplets of high cardinality follows
our results \cite{P13, GP14,LP14,LP16} on high cardinalities of the set of connected components
of moduli spaces of surfaces isogenous to a product (SIP) and exploits the 
solution of the Chisini conjecture by Kulikov \cite{K08} for generic branched
projections.

Putting these into the perspective of the action of the absolute Galois group on varieties
defined over number fields, we are going to show a few new results.
The first concerns a faithful action of the absolute Galois group on
loci of plane curves.

\begin{theo}\label{thm main 1}
There is a faithful action of the absolute Galois group $\Gal$ on the set of connected components of equisingularity strata of plane irreducible curves with at most nodal and ordinary cuspidal singularities. 
\end{theo}

In particular we can derive from that set-up the existence of arbitrarily large arithmetic Zariski multiplets of plane curves.
Indeed, we are able to make use of specific examples of large Galois orbits of surfaces for the final result.
 
\begin{theo}\label{thm main 3}
There is an infinite sequence of prime numbers $p_i$ such that there is a Galois orbit of lengths $\lambda(p_i)=\varphi(p^2_i-1)/8$ of branch curves $B_i$ of degree 
\[
\deg(B_i)=60p_i\left(\frac{p_i-1}{6}-1\right)\left(\frac{p_i+1}{4}-1\right).
\]
In particular the corresponding sequence of $\lambda(p_i)$
is not bounded above.
\end{theo}

\noindent
Let us now explain the organization of this paper.

In the next two sections, addressing preliminary results, we lay the foundational tools and concepts necessary for our results. First, we define, for each automorphism \(\sigma\) of the field of complex numbers \(\mathbb{C}\), a functor \(F_{\sigma}\) from the category of complex varieties to itself. This functor will serve as the central tool in our study of the action of the absolute Galois group on equisingular strata of plane curves. Second, we introduce the notion of generic \(m\)-canonical branch curves. These curves arise as branch curves of a covering \(p\colon S \to \mathbb{P}^2\), where \(S\) is a surface of general type. The covering \(p\) is constructed as the composition of the \(m\)-canonical embedding of \(S\) and a projection onto \(\mathbb{P}^2\) from a generic disjoint center. We then recall the definition and basic properties of surfaces isogenous to a higher product and the weak rigidity theorem of Catanese \cite{cat00}, which allows us to describe the connected components of the moduli space of such surfaces. 

In the final section we first
demonstrate that the associated generic \(m\)-canonical branch curves of surfaces isogenous to a product with different fundamental group form a Zariski pair (see Theorem \ref{theo_ZP}) and, using, in particular cases, the functor \(F_{\sigma}\), even arithmetic Zariski pairs.
Finally, we provide detailed proofs of the main theorems stated in the introduction.

\textbf{Acknowledgement} The authors wish to thank F. Catanese, F. Flamini for fruitful discussions. 
The second author was partially supported by the Research Project PRIN 2020 - CuRVI, CUP J37G21000000001, the MIUR Excellence Department of Mathematics, University of Genoa, CUP D33C23001110001 and INDAM-GNSAGA.



\section{Preliminaries}

First, we investigate the action of $\Aut(\CC)$ and $\Gal$ on equisingular strata of the parametrizing space of plane curves with only nodes and cusps as singularities. Most of the material in this section is well known therefore we shall omit some details. For example a consequence of the Zorn Lemma is the following theorem.

\begin{theo} Every $\sigma \in \Gal$ extends to an automorphism of the field of complex numbers $\CC$. Hence we have a surjective morphism of groups
\[
\varphi \colon \Aut(\CC) \lra \Gal.
\] 
\end{theo}

Following \cite[Def.4.5]{BCG15} we consider functors on the category 
$\underline{\mathbf{\CC-var}}$ of complex varieties.

\begin{defin}\label{functor_Fsigma}
For each   $\sigma \in Aut(\CC)$ let 
\begin{eqnarray*}
F_{\sigma}\colon \underline{\mathbf{\CC-var}} &\lra & \underline{\mathbf{\CC-var}} \\
X &\mapsto & X^{\sigma} \\ 
\left( f \colon X \lra Y\right) & \mapsto &  \left(f^{\sigma}\colon X^{\sigma}\lra Y^{\sigma}\right)
\end{eqnarray*} 
the functor that to a complex variety $X$ associates $X^{\sigma}=X \otimes_{\CC,\sigma} \CC$. 
\end{defin}

The functor we have just defined shares the following properties (see \cite[section 4]{BCG15}):

\begin{rem}
\begin{enumerate}
\item
$\sigma\in \Aut(\CC)$ acts on $\CC[z_0,\dots,z_n]$ by sending the element 
$P(z) = \sum_{I=(i_0,\dots,i_n)} a_Iz^I$ to
\[
\sigma(P)(z) \quad := \quad
\sum_{I=(i_0,\dots,i_n)}  \sigma(a_I)z^I.
\]
\item
Let $X$ be a projective variety
\[
X \quad \subset \quad \PP^n_\CC,\quad X := \{ z \mid f_i(z) = 0 \; \forall i \,\}.
\]
The action of $\sigma$ extends coordinatewise to $\PP^n_\CC$ and carries $X$ to the set
$\sigma(X)$ which is another variety that coincides with $X^\sigma$, the \emph{conjugate
variety}.\\
In fact, since $f_i(z)=0$ if and only if $\sigma(f_i)(\sigma(z))=0$, one has that
\[
X^\sigma \quad = \quad \{ w \mid \sigma(f_i)(w) =0 \; \forall i \,\}.
\]
    \item If $X$ is defined over a subfield $k_0 \subset \mathbb{C}$, i.e., there is a $k_0$-scheme $X_0$ such that $X \cong X_0 \otimes_{k_0} \mathbb{C}$, then $X^\sigma$ depends only on the restriction $\sigma|_{k_0}$ and if moreover $\sigma$ is the identity on $k_0$, then $X^\sigma$ is canonically isomorphic to $X$; we will use this quite often with $k_0=\QQ$.
    
    \item The formation of $X^\sigma$ is compatible with products;
    
    \item For a group action $G \times X \to X$ with $G$ finite (hence defined over $\mathbb{Q}$ as an algebraic group), then $G^\sigma = G$ canonically and applying the conjugation functor gives a conjugate action $G \times X^\sigma = G^\sigma \times X^\sigma \to X^\sigma$.
\end{enumerate}
\end{rem}

In \cite[section 5]{BCG15} the authors argue that
the action of $\Aut(\CC)$ on all complex surfaces descents to an action of $\Aut(\CC)$ on the concected components of the  moduli space of complex projective surfaces of general type. Indeed if $S_1$ and $S_2$ are two surfaces  belonging to the same family $\mathcal{X} \lra B$ we can apply the functor $F_{\sigma}$ and get the following functorial diagram
 \begin{equation}
 \label{functorfamily}
 \xymatrix{
\mathcal{X} \ar[r]^{F_{\sigma}} \ar[d]_{\pi} & \mathcal{X}^{\sigma} \ar[d]^{\pi^{\sigma}} \\
B \ar[r]_{F_{\sigma}} & B^{\sigma}
}
\end{equation}
 and the family $\mathcal{X}^{\sigma} \lra B^{\sigma}$ connects the surfaces $S_1^{\sigma}$ and $S_2^{\sigma}$. More precisely, choose the canonical model $X$ of a surface of general type $S$ and apply the field automorphism $\sigma \in  \Aut(\CC)$ to a point of the Hilbert scheme corresponding to the $m$-canonical image of $S$ .
We obtain a surface $X^{\sigma}$, and whose minimal model is $S^{\sigma}$.  Since the Hilbert scheme  corresponding to $m$-pluricanonical embedded surfaces is defined over $\ZZ$, it is invariant under the functor $F_{\sigma}$ so its connected components are defined over $\overline{\QQ}$. We conclude that the elements  in  the kernel of $\varphi \colon \Aut(\CC) \lra \Gal$ act trivially on connected components of the moduli space of surfaces which is in bijection with the set of connected components of the Hilbert subscheme just defined. For more details see \cite[Proposition 5.7]{BCG15}.
\bigskip

We want to do something similar for parameter spaces of plane curves.
More precisely, to replace the Hilbert scheme we consider equisingularity strata 
of plane curves with a fixed combinatorial type. Like the Hilbert scheme, such strata are in
general neither irreducible nor connected.

\begin{defin}
Given a \emph{combinatorial type of degree $d$ plane curves}
as defined in full generality in \cite{ACT10},
the corresponding \emph{equisingular stratum} in the parameter space of degree $d$ plane curve is given by the locus 
of all curves with that combinatorial type.
\end{defin}
Below we will only consider irreducible plane curves with ordinary
nodes and cusps with topological type determined by the respective
Milnor numbers $\mu=1$ and $\mu=2$.
In this situation, the combinatorial datum of $B$ is simply 
encoded by the triple \((d,n,c)\), where 
\(d = \deg B\) is the degree, \(n\) is the number of nodes, and \(c\) is the number of cusps of \(B\).
\medskip

The corresponding equisingular stratum is a quasi-projective variety 
\[
  \mathcal{S}^{d}_{n,c} \subset \PP^{\frac{d(d+3)}{2}}=:\PP,
\]
as proved by Wahl in~\cite{W74}.  
In particular, \(\mathcal{S}^{d}_{n,c}\) is a \emph{locally closed} subset of a projective space,
and hence a complex algebraic variety endowed with the induced standard analytic topology.
Since such spaces are locally path connected, the connected components of 
\(\mathcal{S}^{d}_{n,c}\) coincide with its path-connected components.
Consequently, \(\mathcal{S}^{d}_{n,c}\) has finitely many (path-)connected components.

Explicitly, we can describe
\[
  \mathcal{S}^{d}_{n,c}
  = 
  \Bigl\{
    [p] \,\Bigm|\,
    p \in \CC[x,y,z]_d, \;
    C_p = \{p=0\} \subset \PP^2 
    \text{ is irreducible with } n \text{ nodes and } c \text{ cusps}
  \Bigr\}.
\]

We observe that if \((B_1,B_2)\) is a Zariski pair, 
where each \(B_i\) is given by an equation \(\{p_i=0\}\), \(i=1,2\), 
then the corresponding points \([p_1]\) and \([p_2]\) lie in 
different path-connected components of \(\mathcal{S}^{d}_{n,c}\).  
The converse is not known in general. 

This motivates the following definition where, by abuse of notation, 
we shall consider a curve \(B_i\) also as element in  $\mathcal{S}^{d}_{n,c}$.



\begin{defin}
Two curves \(B_1, B_2\) in $\mathcal{S}^{d}_{n,c}\) are said to be 
\emph{rigidly isotopic} if there exists a continuous path (in the standard analytic topology)
\(\gamma : [0,1] \to \mathcal{S}^{d}_{n,c}\) 
such that \(\gamma(0) = [p_1]\) and \(\gamma(1) = [p_2]\), 
where \(B_i = \{p_i = 0\}\) for \(i = 1,2\).
Equivalently, \(B_1\) and \(B_2\) belong to the same 
path-connected component of the equisingular stratum 
\(\mathcal{S}^{d}_{n,c}\).
\end{defin}

\begin{rem}
To get this definition into the right perspective let us note
the following:
\begin{enumerate}
\item In \cite{ACT10}, rigid isotopy is defined using smooth paths.
We observe that, in this context, the two definitions -- using either continuous paths or smooth paths -- define the same path components, hence the same equivalence relation.
\item From an algebraic viewpoint, such a path corresponds to a 
flat projective family of plane curves 
\(\pi : \mathcal{X} \to T\) over a connected complex curve \(T\),
whose fibers \(\mathcal{X}_t \subset \PP^2\) all have degree \(d\) 
and precisely \(n\) nodes and \(c\) cusps.
Thus, rigid isotopy classes coincide with the connected components of 
the equisingular stratum, both topologically and algebraically.
\item Consequently, the connected components of the equisingular stratum 
\(\mathcal{S}^{d}_{n,c}\) are precisely the rigid isotopy classes of curves.
\end{enumerate}
\end{rem}

\medskip

We now unpack the last two remarks more carefully, since they involve both the
analytic and the Zariski topologies.

Let \([p_1],[p_2]\in \mathcal{S}^{d}_{n,c}\subset\PP^{\frac{d(d+3)}{2}}\).
We define the following equivalence relations reflecting the two 
topologies:

\begin{enumerate}
\item 
\([p_1]\sim_1 [p_2]\) if they lie in the same path-connected component
of \(\mathcal{S}^{d}_{n,c}\) in the standard analytic topology;

\item 
\([p_1]\sim_2 [p_2]\) if there exists a finite chain of quasi-projective
irreducible curves \(T_1,\dots,T_m\subset \mathcal{S}^{d}_{n,c}\) such that
\(T_i\cap T_{i+1}\neq\varnothing\) for all \(i\)
and \([p_1], [p_2]\) are points on $T_1$ resp.\ $T_m$.\\
(Equivalently, \(\sim_2\) is the transitive closure of the relation
“lies on a common irreducible quasi-projective curve contained in 
\(\mathcal{S}^{d}_{n,c}\)”). 
\end{enumerate}

\noindent
Then these two equivalence relations coincide:
\[
[p_1]\sim_1 [p_2]
\quad\Longleftrightarrow\quad
[p_1]\sim_2 [p_2].
\]

Indeed, the implication from \(\sim_2\) to \(\sim_1\) is clear, since each
irreducible algebraic curve \(T_i\) is path connected with the analytic
topology.\\
For the converse, assume \([p_1]\sim_1 [p_2]\) and let 
\(\gamma:[0,1]\to \mathcal{S}^{d}_{n,c}\) be a path
joining them, continuous in the analytic topology of \( \mathcal{S}^{d}_{n,c} \). Denote by $S_1,\dots,S_\ell$ the irreducible components
of \( \mathcal{S}^{d}_{n,c} \) with respect to the Zariski topology.
Then for every path $\gamma$ as above, there exist
a sequence $i_1,\dots, i_k$ of elements in \(\{1,\dots,\ell\}\)
and real numbers $0=t_0<t_1<\dots<t_k=1$
such that $\gamma([t_{j-1},t_j])\subset S_{i_j}$.

By induction on $k\geq1$ the claim then follows:
If $k=1$, both \([p_1], [p_2]\) belong to the same $S_{i_1}$
and an irreducible algebraic curve exists on $S_{i_1}$ which
contains both.

If $k>1$ then by induction hypothesis both \([p_1],\gamma(t_{k-1})\) 
belong to a connected union of algebraic curves.
Since \(\gamma(t_{k-1}), [p_2]\) both belong to the same component
$S_{i_k}$ the inductive claim also follows,
thus  \([p_1]\sim_2 [p_2]\).

\medskip

Moreover, if \([p_1]\sim_2 [p_2]\) via a connected quasi-projective curve 
\(T\subset\mathcal{S}^{d}_{n,c}\),
we can construct the corresponding flat family as follows:
consider the universal family
\(\mathcal{C}\subset \PP^2\times \PP^{\frac{d(d+3)}{2}}\)
of irreducible degree-\(d\) curves,
and let
\begin{equation}\label{eq_flatfamily}
\pi:\mathcal{X} := \mathcal{C}\times_{\PP^{\frac{d(d+3)}{2}}} T \longrightarrow T
\end{equation}
be the induced family.
Then \(\pi\) is a flat projective family of plane curves of degree \(d\)
with exactly \(n\) nodes and \(c\) cusps as singularities.

In a similar way as for surfaces we use the diagram \eqref{functorfamily} 
in this context, in particular we refer to the map \eqref{eq_flatfamily} and get:
 
\begin{prop} For each $\sigma \in \Aut(\CC)$ the functor $F_{\sigma}$ descends to a well defined map on rigid isotopic classes of plane curves. 
\end{prop}

By this proposition we can define an action of $\Aut(\CC)$ on the set of all rigid isotopic classes of plane curves.  

A consequence of the Theorem \cite[Theorem 11]{PR20} of Parusinski--Rond is that
each such class contains a curve defined over $\bar\QQ$, and we get the following result

\begin{prop} For each $\sigma$ in the kernel of $\Aut(\CC)\to \Gal$ the functor $F_{\sigma}$ preserves all rigid isotopy classes of irreducible plane curves.

Hence $\Gal$ acts on the rigid isotopy classes.
\end{prop}

\begin{proof} 
The image under 
$F_{\sigma}$ of the rigid isotopy class of a curve $B$ is the isotopy class of the image of any curve rigidly isotopic to $B$ under 
$F_{\sigma}$. Since $\sigma$ is in the kernel of $\Aut(\CC)\to \Gal$, this functor
fixes all curves defined over $\bar\QQ$. Now every isotopy class contains such curve
by  \cite[Theorem 11 ii) \& iii)]{PR20}, $F_{\sigma}$, hence $\sigma$ fixes every isotopy class.
Thus the $\Aut(\CC)$-action descends to a $\Gal$-action on isotopy classes.
\end{proof}


We want to link more tightly the action of $\Gal$ on the connected components of moduli space of surfaces with the action on the  connected components of equisingular strata of plane curves with only nodes and cusps as singularities.  To do so we have to recall the general theory of branched curves, this is the same strategy we used in our previous paper \cite{LP20}. 

\section{branch curves and surfaces isogenous to a product}

There is natural way to produce many singular plane curves with the above mentioned singularities. Indeed, it is enough to consider branch curves of coverings of $\PP^2$. To give such covering we will proceed in the following way: first we consider a surface of general type $S$ with ample canonical class, then we consider the natural immersion in a $\PP^n$ by a multicanonical system. Finally, we project the image generically  to $\PP^2$. This yields a covering $S \longrightarrow \PP^2$. To be more precise let us explain in details this procedure.  

\begin{defin}[cf. \cite{K99}]
\label{kulikov}
Let $B\subset \PP^2$ be an irreducible plane algebraic curve with ordinary cusps
and nodes as the only singularities. The curve $B$ is called \emph{generic branch curve} if there is a finite morphism $p:S\to \PP^2$ with $\deg p\geq3$ such that
\begin{enumerate}
\item
$S$ is a smooth irreducible projective surface,
\item
$p$ is unramified over $\PP\setminus B$,
\item
$p^*(B) = 2R + C$, where $R$ is a smooth irreducible reduced curve
and $C$ is a reduced curve,
\item
the morphism $p_{|R}:R\to B$ coincides with the normalization of $B$.
\end{enumerate}
In this case $p$ is called a \emph{generic covering of the projective plane}.
\end{defin}



\begin{theo}\cite[Theorem 1.1]{CiF11} \label{theo_CiroFlam}
\label{cilibertoflamini}
Let $S \subset \PP^r$ be a smooth, irreducible, projective surface. Then the ramification curve on $S$ of a generic projection of $S$ to $\PP^2$ is smooth and irreducible and the branch curve in the plane is also irreducible and has only nodes and cusps, respectively, corresponding to two simple ramification points and one double ramification point.
\end{theo}

In particular such generic projection of $S$ onto $\PP^2$ has a generic branch curve in the sense of Def. \ref{kulikov}.

\begin{defin} A plane curve $B$ is called a \emph{generic m-canonical branch curve} if there exists a smooth surface $S$ with $mK_S$ very ample and a commutative diagram
\[
\begin{xy}
\xymatrix{
 S  \ar[rr]^{\phi_m} \ar[ddrr]_{p} & & \PP^{P_m-1}  \ar@{-->}[dd]
\\ 
\\
& & \PP^2
  }
\end{xy}
\]
where $\phi_m$ is the $m$-canonical map,  
$p$ is a generic covering and $B$ is the branch locus of $p$.
\end{defin}

We are now able to associate to a surface of general type a plane curve with only nodes and cusps as singularities. 
Among all the surfaces of general type there are some for which it is possible to control accurately the $\Gal$  action on the connected components of their moduli space. These surfaces are the so called \emph{surfaces isogenous to a product} (or SIPs for short). They were introduced by Catanese in \cite{cat00} and the action of  $\Gal$ on their moduli space was investigated in \cite{BCG15}. 

Thanks to their simple definition SIPs are incredibly versatile and they give rise to a large amount of interesting examples.  Let us be more precise:

\begin{defin}\label{def.isogenous} A surface $S$ is said to be \emph{isogenous to a higher product of curves}\index{Surface isogenous to a higher product of curves} if and only if, $S$ is a
quotient $(C_1 \times C_2)/G$, where $C_1$ and $C_2$ are curves of
genus at least two, and $G$ is a finite group acting freely on
$C_1 \times C_2$. (To ease notation we call the surfaces SIP)
\end{defin}
Using the same notation as in Definition \ref{def.isogenous}, let
$S$ be a SIP, and $G^{\circ}:=G
\cap(Aut(C_1) \times Aut(C_2))$. Then $G^{\circ}$ acts on the two
factors $C_1$, $C_2$ and diagonally on the product $C_1 \times
C_2$. If $G^{\circ}$ acts faithfully on both curves, we say that
$S= (C_1 \times C_2)/G$ is a \emph{minimal
realization}. In \cite{cat00}
it is also proven that any
SIP admits a unique minimal realization. 
\medskip

{\bf Assumptions.} In the following we always assume:
\begin{enumerate}
\item Any SIP  $S$  is given by its unique minimal realization;
\item $G^{\circ}=G$, this case is also known as \emph{unmixed type}, see \cite{cat00}.
\end{enumerate}
Under these assumptions we have. \nopagebreak
\begin{prop}~\cite{cat00}\label{isoinv}
Let $S=(C_1 \times C_2)/G$ be a SIP, then $S$ is a minimal surface of general type with the following invariants:
\begin{equation}\label{eq.chi.isot.fib}
\chi(S)=\frac{(g(C_1)-1)(g(C_2)-1)}{|G|},
\quad
e(S)=4 \chi(S),
\quad
K^2_S=8 \chi(S).
\end{equation}
The irregularity of these surfaces is computed by
\begin{equation}\label{eq_irregIsoToProd}
q(S)=g(C_1/G)+g(C_2/G).
\end{equation}
Moreover the fundamental group $\pi_1(S)$ fits in the following short exact sequence of groups
\begin{equation}\label{eq_fundGroupS}
1 \longrightarrow \pi_1(C_1) \times \pi_1(C_2) \longrightarrow \pi_1(S) \longrightarrow G \longrightarrow 1.
\end{equation}
\end{prop}
The most important property of surfaces isogenous to a product is their weak rigidity property.
\begin{theo}~\cite[Theorem 3.3, Weak Rigidity Theorem]{cat04}
\label{weak}
Let $S=(C_1 \times C_2)/G$ be a surface isogenous to a higher
product of curves. Then every surface with the same
\begin{itemize}
\item topological Euler number and
\item fundamental group
\end{itemize}
is diffeomorphic to $S$. The corresponding  moduli space
$\mathcal{M}^{top}(S) = \mathcal{M}^{\it diff}(S)$ of surfaces
(orientedly) homeomorphic (resp. diffeomorphic) to $S$ is either
irreducible and connected or consists of two irreducible connected
components exchanged by complex conjugation.
\end{theo}



We can now link SIP and branch curves as done in \cite{LP20}. In this contest we need to  be more careful in choosing the pluricanonical embeddings because we have no specific restrictions on the group $G$  as we had in \cite{LP20}. To this end we have change the hypothesis of \cite[Lemma 3.6]{LP20} and the claim is proved for the $3$-canonical embedding  instead of $2$-canonical. We get an analogue theorem to \cite[Theorem 3.12]{LP20}.

\begin{theo}\label{theo_ZP} Let $S_1$ and $S_2$ be two surfaces isogenous to a product  with the same Euler number and $\pi_1(S_1) \neq \pi_1(S_2)$. Then for $m \geq 3$  the corresponding generic m-canonical branch curves $(\PP^2, B_1)$, $(\PP^2, B_2)$ are a Zariski pair.
\end{theo}

Now that we have a link between Zariski pairs and SIPs we want to add arithmetic information to this setting.
In \cite{BCG15} the three authors exploited the weak rigidity theorem to prove the following result.

\begin{theo}{\cite[Theorem 1.1]{BCG15}}\label{theo_iso} If $\sigma \in \Gal$ is not in the conjugacy class of the complex conjugation, then there exists a surface isogenous to a product $X$ such that $X$ and the Galois conjugate varieties $X^{\sigma}$ have non-isomorphic fundamental group. 
\end{theo}

Moreover, they could prove even a much stronger theorem.

\begin{theo}{\cite[Theorem 1.3]{BCG15}}\label{theo_ffaGal} The absolute Galois group $\Gal$ acts faithfully on the set of connected components of the (coarse) moduli space of surfaces of general type. 
\end{theo}


We conclude this preliminary section with a consideration of numerical invariants. Since these numerical invariants of a SIP are completely determined by the Euler number, which is a fortiori determined by the genera of the curves involved and by the order of the group $G$, we are able to determine the degree $d$, the number $n$ of nodes, and the number $c$ of cusps of the branch curve.
Indeed repeating the same argument as for the proof of \cite[Lemma 3.10]{LP20} we have:

\begin{lem}\label{lem_3caninv} Let $(\PP^2,B)$ be a generic 3-canonical branch curve of regular surfaces isogenous to a product $S$  then
\begin{equation}\label{eq_deg} d=\deg B=30c_1(S)^2,
\end{equation}
\begin{equation}\label{eq_cusps} c=137c_1(S)^2-c_2(S)
\end{equation}
\begin{equation}\label{eq_nodes} n=450(c_1(S)^2)^2-237c_1(S)^2+c_2(S).
\end{equation}
\end{lem}


\section{Proof of the main theorems}

To prepare the proofs we combine
the results in Theorem \ref{theo_ZP} and Theorem  \ref{theo_iso} to give examples of arithmetic Zariski pairs as follows.

\begin{theo}\label{theo_AZP} Let $S$ and $S^{\sigma}$ be two surfaces isogenous to a product as in  Theorem \ref{theo_iso} then they have  the same Euler number but $\pi_1(S) \neq \pi_1(S^{\sigma})$. Then for $m \geq 3$  the corresponding generic m-canonical branch curves $(\PP^2, B)$, $(\PP^2, B^{\sigma})$ are arithmetic Zariski pair.
\end{theo} 

\begin{proof} By  Theorem \ref{theo_ZP}  $(\PP^2, B)$, $(\PP^2, B^{\sigma})$ are  Zariski pair. We only have to prove that the pair is arithmetic. By the proof of \cite[Theorem 1.3]{BCG15}  the field of moduli is a number field.
Also a field of definition is a number field, though a bigger one in general,  therefore we know that the $m$-canonical model $X$ of $S$ and $X^{\sigma}$ for $S^{\sigma}$ in $\PP^{P_m-1}$ is defined over $\overline{\QQ}$. By \cite[Theorem 1.1]{CiF11}  we can map $S$ 
to $\PP^2$ by a generic projection $\pi$ with branch curve $B$ 
that is irreducible with at most nodes and ordinary cusps as singularities and again defined over a number field. In particular, this allows us to apply the functor $F_{\sigma}$ defined in \ref{functor_Fsigma} and to obtain the following diagram.
\[
\begin{xy}
\xymatrix{
 X  \ar[dd]_{\pi} & & X^{\sigma}  \ar[dd]^{\pi^{\sigma}}
\\
 \ar[rr]^{F_{\sigma}} & &
\\
B \subset \PP^2& & B^{\sigma} \subset \PP^2
  }
\end{xy}
\]
This completes the proof, since also $\pi^\sigma$ is a generic projection and 
$B^\sigma$ is a generic branch curve.
 \end{proof}

We are now in the position to prove Theorems \ref{thm main 1} and \ref{thm main 3}, that we recall for reader's convenience. 

\begin{theo}  
There is a faithful action of the absolute Galois group $\Gal$ on the set of connected components of equisingularity strata of plane irreducible curves with at most nodal and ordinary cuspidal singularities. 
\end{theo}

\begin{proof}
For every $\sigma \in \Gal$  non trivial, there exists a SIP $X$ defined over $\overline{\QQ}$ such that $X$ and $X^{\sigma}$ have non isomorphic topological fundamental group. This follows from Theorem \ref{theo_iso}. 

We can embed $X$ by the $3-$canonical map into a projective space then consider a linear projection $\pi$ from a generic center of projection to $\PP^2$.  This  gives a branch curve $B$, that is irreducible with at most nodes and ordinary cusps as singularities by Theorem \ref{theo_CiroFlam}. 

We extend  $\sigma$ to an automorphism of $\CC$ to get another projection $\pi^{\sigma}\colon X^{\sigma} \longrightarrow \PP^2$ branched along $B^{\sigma}$.
Then by our previous result \cite{LP20} $(\PP^2,B)$ and $(\PP^2,B^{\sigma})$ is a Zariski pair by Theorem \ref{theo_AZP} . We deduce that $B$ and $B^{\sigma}$ are in different strata and thus $\sigma$ does not act trivially on the set of strata.   
\end{proof}

Finally we use the article \cite{GJT18} to give more concrete examples of larger multiplets. Their construction involves Beauville surfaces, which are rigid SIP
and their result is stated in terms of the Euler totient function $\varphi$:

\begin{theo}
There is an infinite sequence of prime numbers $p_i$ such that there is a Galois orbit of lengths $\lambda(p_i)=\varphi(p^2_i-1)/8$ of branch curves $B_i$ of degree 
\[
\deg(B_i)=60p_i\left(\frac{p_i-1}{6}-1\right)\left(\frac{p_i+1}{4}-1\right).
\]
In particular the corresponding sequence of $\lambda(p_i)$ is not bounded above.
\end{theo}
\begin{proof}
By Dirichlet the sequence of prime numbers $p$ such that $p \equiv 19 \, (\mod 24)$ is not bounded. Since $\varphi$ takes every integer value at most a finite number of times the last claim is immediate.

For each one of these primes $p$, we know by Theorem 7.1 \cite{GJT18} there exists a $\Gal$ orbit of Beauville surfaces with branching data
\[
(2,3,p-1), \quad (2,4,p+1) \quad \textrm{ and group}\quad \PGL(2,p)
\]
of length 
\[
\lambda(p)=\frac{1}{8} \cdot \varphi\left(p^2-1\right).
\]

We use the formula given in Lemma \ref{lem_3caninv} to get that the degrees of the corresponding branch curves $B_p$ are
\[
\deg(B_p)=60p\left(\frac{p-1}{6}-1\right)\left(\frac{p+1}{4}-1\right).
\]
To get a Galois orbit of the same length we have to show that the branch curves of two surfaces in different components are mapped to different components of strata. 

Assume the contrary that they belong to the same stratum then by our result \cite[Theorem 3.12]{LP20} also the surfaces have isomorphic fundamental group, if so by the weak rigidity theorem they are in the same component contradicting our assumption. 
\end{proof}


%

\begin{thebibliography} {9}
%
\bibitem[A94] {artal94}
E. Artal.
{\it Sur les couples de Zariski}. 
J. Algebraic Geom. 3(2):223-247, 1994.
%
\bibitem[AC17]{AC17}
E.Artal, J. Cogolludo,  {\it Some Open Questions on Arithmetic Zariski Pairs}. In Singularities in Geometry, Topology, Foliations and Dynamics. Trends in Mathematics. Birkhäuser (2017).
%
%
\bibitem[ACT10] {ACT10}
E. Artal, J.I. Cogolludo, and H. Tokunaga
{\it A survey on Zariski pairs}. 
In Algebraic geometry in East Asia--Hanoi 2005, volume 50 of Adv. Stud. Pure Math. pages 1-100. Math. Soc. Japan, Tokyo, (2008).
%
%
\bibitem[BCG06] {BCG06}
I. Bauer, F. Catanese, F. Grunewald, \textit{Chebycheff and Belyi
polynomials, dessins d'enfants, Beauville surfaces and group
theory}. Mediterr. J. Math. \textbf{3}, (2006), 121--146.
%
\bibitem[BCG15] {BCG15}
I. Bauer, F. Catanese, F. Grunewald, \textit{Faithful actions of the absolute Galois group on connected components of moduli spaces}. Invent. Math., 199, (2015), pp. 859--888.
%
\bibitem[B73]{B73} 
E. Bombieri, {\em Canonical models of surfaces of general type}. Publications Mathematiques de L'IHES {\bf 42} (1973), 171--219.
%
%
%
\bibitem[Cat00]{cat00} 
F. Catanese, \textit{Fibred surfaces, varieties isogenous to a
product and related moduli spaces}. Amer. J. Math. \textbf{122},
(2000), 1--44.
%
\bibitem[Cat03]{cat04}
F, Catanese, \textit{Moduli spaces of surfaces and real structures}. Ann. of Math.
\textbf{158}, (2003), 577--592.
%
%
\bibitem[CiF11] {CiF11}
C. Ciliberto, F. Flamini, \textit{On the branch curve of a general projection of a surface to a plane}.
Trans. Amer. Math. Soc., {\bf 363}, (2011), 3457--3471.
%
%
\bibitem[GP14]{GP14}
S. Garion, M. Penegini, \textit{Beauville surfaces, moduli spaces and finite groups}. Comm. Algebra  {\bf 42}, (2014), 2126--2155.
%
%
%
\bibitem[GJ15]{GJT18} 
G. Gonzalez, A. Jaikin, \textit{The absolute Galois group acts faithfully on regular dessins and on Beauville surfaces}. Proc.\ LMS {\bf 111} (2015),  775--796.
%
\bibitem[GJT18]{GJT18} 
G. Gonzalez, G. Jones, D. Torres, \textit{Arbitrarily large Galois orbits of non-homeomorphic surfaces}. European Journal of Mathematics {\bf 4} (2018),  223--241.
%
\bibitem[H86] {H86}
J. Harris, \textit{On the Severi problem}. Invent. Math. \textbf{84} (1986),  445--461. 
%
\bibitem[K99]{K99}
Vik. S. Kulikov, \textit{On Chisini's conjecture}. Izv. Math., 63:6 (1999), 1139--1170.
%
\bibitem[K08]{K08}
Vik. S. Kulikov, \textit{On Chisini's conjecture. II}, Izv. Math., 72:5 (2008), 901--913.
%
\bibitem[LP14]{LP14}
M. L\"onne, M. Penegini, \textit{On asymptotic bounds for the number of irreducible components of the moduli space of surfaces of general type}. 
Rend.\ Circ.\ Mat.\ Palermo (2) \textbf{64} (2015), 483--492.
%
\bibitem[LP16]{LP16}
M. L\"onne, M. Penegini, \textit{On asymptotic bounds for the number of irreducible components of the moduli space of surfaces of general type II}. 
  Doc. Math. 21 (2016), 197--204. 
%
\bibitem[LP20]{LP20}
M. L\"onne, M. Penegini, \textit{On Zariski multiplets of branch curves from surfaces isogenous to a product}. Michigan Mathematical Journal, 2020, 69(4), pp. 779--792.
%
\bibitem[P13]{P13}
M. Penegini, \textit{Surfaces isogenous to a product of curves, braid groups and mapping class groups} in "Beauville Surfaces and Groups", Springer Proceedings in Math. \ and\ Stats.,\ (2015), 129--148.
%
\bibitem[PR20]{PR20}
A. Parusinski, G. Rond, \textit{Algebraic varieties are homeomorphic to varieties defined over number fields}.
Comment. Math. Helv. {\bf 95} (2020), 339--359.
%
\bibitem[S21] {S21}
F. Severi,  \textit{Vorlesungen \"uber algebraische Geometrie}. Teubner, Leipzig, 1921. 
%
\bibitem[S08]{S08}
I. Shimada, {\it On arithmetic Zariski pairs in degree 6}. Advances in Geometry, {\bf 8}, (2008),  205--225.
%
\bibitem[W74]{W74}
 J. Wahl, \textit{Equisingular deformations of plane algebroid curves}. Trans. Am. Math. Soc. {\bf 193}, (1974),
143--170.
\end{thebibliography}
\end{document}